\documentclass[11pt,a4paper]{smfart} 
\usepackage{mathsmf}
\usepackage{smfthm}
\usepackage{graphicx}
\usepackage{color}
\usepackage[utf8]{inputenc}
\usepackage[T1]{fontenc}
\usepackage[french]{babel}
\usepackage[plainpages]{hyperref}

\title[Surfaces polyédrales tendues et nombre chromatique %
relatif]{Plongements polyédraux tendus et nombre chromatique relatif des 
surfaces à bord}
\alttitle{Tight polyhedral embeddings and relative chromatic number of
surfaces with boundary}
\author{Pierre Jammes}
\address{Univ. Nice Sophia Antipolis, CNRS, LJAD, UMR 7351\\
06100 Nice France}
\email{pjammes@unice.fr}
\date{}
\begin{document}
\begin{abstract}
Le nombre chromatique relatif $c_0(S)$ d'une surface compacte $S$ à bord
est défini comme la borne supérieure des nombres chromatiques des graphes
plongés dans $S$ avec tous leurs sommets sur $\partial S$. Cet invariant
topologique a été introduit pour l'étude de la multiplicité de la première
valeur propre de Steklov sur $S$. Dans cet article, on montre que $c_0(S)$
est aussi pertinent pour l'étude des plongements polyédraux tendus de $S$
en établissant deux résultats.
Le premier est que s'il existe un plongement polyédral tendu de $S$
dans $\R^n$ qui n'est pas contenu dans un hyperplan, alors $n\leq c_0(S)-1$.
Le second est que cette inégalité est optimale pour les surfaces
de petit genre.
\end{abstract}
\begin{altabstract}
The relative chromatic number $c_0(S)$ of a compact surface $S$ with 
boundary is defined as the supremum of the chromatic numbers of graphs 
embedded in $S$ with all vertices on $\partial S$. This topological 
invariant was introduced for the study of the multiplicity of the 
first Steklov eigenvalue of $S$. In this article, we show that $c_0(S)$ is
also relevant for the study of tight polyhedral embeddings of $S$ by
proving two results. 
The first one is that if there is a tight polyhedral embedding of $S$ 
in $\R^n$ which is not contained in a hyperplane, then $n\leq c_0(S)-1$.
The second result is that this inequality is sharp for surfaces of small genus.
\end{altabstract}
\keywords{surfaces polyédrales, nombre chromatique, plongements tendus}
\altkeywords{polyheral surfaces, chromatic number, tight embeddings}
\subjclass{52B70, 57Q35, 53A05, 55A15, 57M15, 05C10}

\maketitle

\section{Introduction}

Le problème de Heawood consiste à déterminer le nombre chromatique
des surfaces closes de genre non nul, c'est-à-dire la borne supérieure 
des nombres
chromatiques des graphes qu'on peut plonger dans cette surface. Son étude
a été amorcée par P.~J.~Heawood \cite{he90} et L.~Heffter \cite{he91},
et sa résolution achevée trois quart de siècle plus tard par G.~Ringel et 
J.~Youngs \cite{ry68}: la valeur du nombre chromatique d'une surface
$S$ de caractéristique d'Euler $\chi$ est $c(S)=\lfloor
\frac{7+\sqrt{49-24\chi}}2\rfloor$, sauf pour la bouteille de Klein
dont le nombre chromatique vaut~6. On peut consulter~\cite{ri74} pour la résolution
complète du problème et un survol historique. 

 Ce nombre chromatique joue un 
rôle dans deux problèmes qui semblent \emph{a priori} sans rapport. 
Le premier concerne les plongements polyédraux substantiels tendus de $S$ 
dans $\R^n$ (voir ci-dessous la définition~\ref{intro:tendu} d'un plongement
tendu ; un plongement dans $\R^n$ est dit \emph{substantiel} si son
image n'est pas contenue dans un hyperplan affine ou, de manière équivalente,
si le sous-espace affine engendré par l'image du plongement est $\R^n$). Il
découle des travaux 
de T.~F.~Banchoff \cite{ba65} et W.~Kühnel \cite{ku80} que la valeur 
maximale de la dimension $n$ telle qu'un tel plongement
existe est $c(S)-1$.

 Le second problème touche au spectre des opérateurs
de Schrödinger sur les surfaces. Y.~Colin de Verdière a conjecturé dans 
\cite{cdv87} que la multiplicité de la deuxième valeur propre 
d'un opérateur de Schrödinger sur $S$ est majorée par $c(S)-1$, et il
montre qu'une telle majoration serait optimale.
Grâce aux travaux successifs de S.~Y.~Cheng \cite{ch76}, G.~Besson 
\cite{be80}, Y.~Colin de Verdière \cite{cdv87} et B.~Sévennec \cite{se94} 
\cite{se02}, cette conjecture est démontrée pour les surfaces closes 
de caractéristique d'Euler supérieure ou égale à -3.

 Récemment, ce problème de multiplicité de valeurs propres a été étudié
pour le spectre de Steklov, qui est le spectre d'opérateurs de type
Dirichlet-Neumann sur les variétés à bord (voir \cite{fs12}, \cite{ja14}, 
\cite{kkp12}). Dans \cite{ja14}, la conjecture de 
Colin de Verdière est étendue au spectre de Steklov en remplaçant le 
nombre chromatique par un autre invariant, spécifique aux surfaces à bord
et baptisé \emph{nombre chromatique relatif} dans \cite{ja13} (voir
la définition~\ref{intro:def2} ci-dessous). Le but du présent article est
de montrer que ce nouvel invariant est pertinent pour l'étude des
plongements polyédraux tendus de surfaces à bord, pour lesquels les
résultats de \cite{ku80} ne sont que partiels.

Rappelons et précisons quelques définitions. Le nombre chromatique
relatif diffère du nombre chromatique usuel par le fait qu'on impose
aux graphes plongés dans la surface d'avoir tous leurs sommets sur le bord :
\begin{definition}\label{intro:def1}
Soient $G$ un graphe fini et $S$ une surface compacte à bord.
Un plongement de $G$ dans $S$ sera appelé \emph{plongement propre}
si ce plongement envoie tous les sommets de $G$ sur $\partial S$.
\end{definition}
\begin{definition}[\cite{ja13}]\label{intro:def2}
Si $S$ est une surface compacte à bord, on appelle \emph{nombre
chromatique relatif} de $S$, noté $c_0(S)$,
la borne supérieure des nombres chromatiques des graphes finis qui admettent
un plongement propre dans $S$. 
\end{definition}
Contrairement au nombre chromatique usuel, la valeur du nombre chromatique 
relatif n'est connue que pour certaines surfaces de petit genre, ou sous
des hypothèses sur le nombre de composantes connexes du bord en genre 
quelconque. On 
rappelera ce qu'on sait à son sujet au paragraphe~\ref{rappel:relatif}.

Les plongements tendus admettent plusieurs définitions qui ne sont
pas équivalentes pour les surfaces à bord:
\begin{definition}\label{intro:tendu}
Un plongement $\varphi:S\to\R^n$ est $p$-tendu si pour tout demi-espace 
affine ouvert $E$, le morphisme d'homologie $H_p(\varphi(S)\cap E)\to 
H_p(\varphi(S))$ à coefficients dans $\mathbb Z_2$ induit par l'inclusion 
est injectif. Le plongement est tendu s'il est $p$-tendu pour tout $p$.
\end{definition}
On s'intéressera en particulier aux plongements tendus et 0-tendus. Ces
deux priopriétés sont équivalentes pour les surfaces closes (\cite{ku95},
lemme~2.5), mais pas
les surfaces à bord (voir les rappels de la section suivante, en particulier
le lemme~\ref{rappel:lem4}). La propriété dite \emph{TPP} (\emph{two pieces
property}), définie pour une partie de $\R^n$ par le fait que tout hyperplan
affine la sépare en au plus deux parties connexes, est équivalente
pour l'image du plongement $\varphi$ à la $0$-tension.
\begin{rema}
Bien que ces propriétés soient généralement définies pour
des plongements euclidiens ou affines, elles sont en fait projectives: 
elles sont préservées par changement de carte affine tant que l'image
du plongement ne rencontre pas l'hyperplan à l'infini.
\end{rema}

Ces précisions étant données, on peut énoncer les deux résultats de
cet article. Le premier majore la dimension de l'espace dans lequel
une surface à bord peut se plonger de manière polyédrale, substantielle
et tendue. La borne obtenue est analogue à celle donnée dans \cite{ba65} 
par T.~F.~Banchoff pour les surfaces closes.
\begin{theo}\label{intro:th1}
Si une surface compacte à bord $S$ admet un plongement polyédral
substantiel et tendu dans $\R^n$, alors $n\leq c_0(S)-1$.
\end{theo}

Le second théorème montre que dans l'inégalité précédente, l'égalité 
est atteinte pour des surfaces de petit genre. En l'état actuel des 
connaissances, on ne peut pas espérer
démontrer l'égalité pour toute surface comme dans le cas des surfaces closes
car le résultat de W.~Kühnel s'appuie de manière cruciale sur la
résolution du problème de Heawood (plus précisément sur l'exhibition
de plongement de graphes complets dans les surfaces). Or, ce problème 
n'est pas encore résolu pour le nombre chromatique relatif.

Pour une surface close $S$ donnée, on notera $S_p$ la surface obtenue
en enlevant $p$ disques disjoints à $S$.
 Rappelons que les travaux de W.~Kühnel (\cite{ku80}, théorème~D), 
reformulés à l'aide du nombre 
chromatique relatif (voir le théorème~\ref{rappel:th} de la section
suivante), donnent l'égalité 
dans le théorème~\ref{intro:th1} pour les
surfaces $S_p$ vérifiant $p=1$ ou
$p\geq c(S)/2$, quelle que soit la surface close $S$.
\begin{theo}\label{intro:th2}
Si $S$ est une surface close orientable de genre inférieur ou égal à 2, ou 
non orientable de genre inférieur ou égal à 3, alors pour tout entier 
$p>0$i, $S_p$ admet un plongement polyédral
substantiel et tendu dans $\R^{c_0(S_p)-1}$ dont tous les sommets sont 
sur $\partial S$.
\end{theo}
Pour les surfaces de genre immédiatement supérieur à ceux couverts par
ce théorème, la connaissance du nombre chromatique relatif commence
à être lacunaire.
\begin{rema}
La contrainte que tous les sommets du plongement soient sur le bord
de la surface n'est pas indispensable pour montrer que le 
théorème~\ref{intro:th1} est optimal, et elle n'était pas présente dans 
les travaux de W.~Kühnel (avec entre autres conséquences qu'on ne peut pas
toujours reprendre telles quelles les constructions de Kühnel pour 
démontrer le théorème~\ref{intro:th2}). Elle est toutefois motivée d'une 
part par l'analogie avec la définition~\ref{intro:def1}, et d'autre part par
le fait que sous cette hypothèse les propriétés de tension et 0-tension
sont équivalentes (voir lemme~\ref{rappel:lem4}).
\end{rema}

On commencera par rappeler dans la section~\ref{rappel} quelques résultats 
sur le nombre chromatique relatif et les plongements polyédraux tendus.
La section~\ref{borne} sera consacrée à la démonstration du 
théorème~\ref{intro:th1}. Enfin, le théorème~\ref{intro:th2} fera l'objet
de la section~\ref{cons}.

\section{Rappels}\label{rappel}
\subsection{Le nombre chromatique relatif}\label{rappel:relatif}
On rassemble ici les résultats connus sur le nombre chromatique 
relatif d'une surface compacte à bord. Rappelons qu'on note $c(S)$ 
le nombre chromatique usuel d'une surface close, $c_0(S)$ le nombre
chromatique relatif d'une surface à bord, $\chi(S)$ la caractéristique
d'Euler d'une surface compacte (avec ou sans bord), et que si $S$ est close,
on note $S_p$ la surface à bord obtenue en enlevant $p$~disques disjoints 
à~$S$.

Le théorème qui suit donne les résultats généraux sur $c_0$, notamment
l'encadrement de $c_0(S_p)$ en fonction de $c(S)$.
\begin{theo}[\cite{ja13}, théorème~3.1]\label{rappel:th}
Le nombre chromatique relatif $c_0(S_p)$ est une
fonction croissante de $p$ et vérifie les inégalités
\begin{equation}
c(S)-1\leq c_0(S_p)\leq
\inf\left(c(S),\frac{5+\sqrt{25-24\chi(S)+24p}}2\right).
\end{equation}
En outre,  $c_0(S_p)$ est le nombre de sommets du plus
grand graphe complet proprement plongeable dans $S_p$, et on a 
$c_0(S_1)=c(S)-1$ et
$c_0(S_p)=c(S)$ si 
$p\geq(c(S)-1)/2$.
\end{theo}
Dans \cite{ja13} sont aussi calculées les valeurs particulières du
nombre chromatique relatif sur les surfaces de petit genre.
Les tables~\ref{rappel:table1} et~\ref{rappel:table2} ci-dessous rassemblent 
les valeurs connues du nombre chromatique relatif (le numéro de la ligne
correspondant au nombre $p$ de composantes de bord). 

\begin{table}[h]
\begin{center}
\renewcommand{\arraystretch}{1.5}
\begin{tabular}{r||c|c|c|c|c|}
&$\mathbb S^2$&$\mathbb T^2$&$\mathbb T^2\#\mathbb T^2$&$\#3\mathbb T^2$&
$\#4\mathbb T^2$\\\hline
1 & 3 & 6 & 7 & 8 & 9\\\hline
2 & 4 & 6 & 8 & ? & 9\\\hline
3 & 4 & 7 & 8 & 9 & 10\\\hline
4 & 4 & 7 & 8 & 9 & 10 \\\hline
\end{tabular}
\bigskip
\caption{Nombre chromatique relatif des surfaces orientables
de petit genre}
\label{rappel:table1}
\end{center}
\end{table}
\begin{table}[h]
\begin{center}
\renewcommand{\arraystretch}{1.5}
\begin{tabular}{r||c|c|c|c|c|c|c|c|c|c|}
&$\mathbb P^2$&$\mathbb K^2$&$\#3\mathbb P^2$&$\#4\mathbb P^2$&
$\#5\mathbb P^2$&$\#6\mathbb P^2$&$\#7\mathbb P^2$&$\#8\mathbb P^2$&
$\#9\mathbb P^2$\\\hline
1 & 5 & 5 & 6 & 7 & 8 & 8 & 9 & 9 & 9\\\hline
2 & 5 & 6 & 7 & ? & 8 & ? & 9 & ? & ?\\\hline
3 & 6 & 6 & 7 & 8 & 9 & 9 & 9 & ? & 10 \\\hline
4 & 6 & 6 & 7 & 8 & 9 & 9 & 10 & 10 & 10 \\\hline
\end{tabular}
\renewcommand{\arraystretch}{1}
\bigskip
\caption{Nombre chromatique relatif des surfaces non orientables
de petit genre}
\label{rappel:table2}
\end{center}
\end{table}

\subsection{Quelques propriétés des plongements polyédraux tendus}
On va rappeler ici divers résultats, établis essentiellement par 
T.~Banchoff et W.~Kühnel, sur les plongements polyédraux tendus 
de surfaces. Certains concernent aussi des graphes, pour lesquels la
propriété de 0-tension (ou de manière équivalente la \emph{TPP}) est bien
définie ; cette propriété implique en particulier que l'image de chaque arête 
du graphe par un plongement 0-tendu est un segment.
 
La démonstration des cinq lemmes qui suivent est rappelée dans 
le chapitre~2 de~\cite{ku95}. Pour ne pas alourdir les énoncés, on
identifiera les surfaces et les graphes à leur plongement.

Les deux premiers lemmes lient la 0-tension d'une surface polyédrale
à un graphe, le squelette de son enveloppe convexe dans le premier et
son propre squelette dans le second.
\begin{lemm}[\cite{ba65}, lemme~3.1]\label{rappel:lem1}
Une surface polyédrale 0-tendue contient le 1-squelette de son enveloppe
convexe. 
\end{lemm}
\begin{lemm}[\cite{ku80}, lemme~1]\label{rappel:lem2}
Une surface polyédrale à faces convexes est 0-tendue si et seulement si
son 1-squelette est 0-tendu.
\end{lemm}

L'énoncé qui suit permet de caractériser les graphes 0-tendus.
\begin{lemm}[\cite{ku80}, lemme~3]\label{rappel:lem3}
Un graphe $G$ est 0-tendu si et seulement s'il vérifie les deux conditions
qui suivent :
\begin{enumerate}
\item[(i)] le graphe $G$ contient le 1-squelette de son enveloppe convexe;
\item[(ii)] tout sommet de $G$ qui n'est pas un point extrême de 
l'enveloppe convexe de $G$ est contenu dans l'enveloppe convexe de ses
voisins.
\end{enumerate}
\end{lemm}

Un autre résultat de W.~Kühnel précise la distinction entre plongements
tendus et 0-tendus d'une surface. C'est ce lemme qui garantira que 
les plongements du théorème~\ref{intro:th2} sont tendus.
\begin{lemm}[\cite{ku80}, lemme~7]\label{rappel:lem4}
Pour toute surface polyédrale $S$ à bord, les propositions suivantes sont 
équivalentes:
\begin{enumerate}
\item[(i)] la surface $S$ est tendue;
\item[(ii)] la surface $S$ est 0-tendue et tous les points extrêmes de son
enveloppe convexe sont situés sur son bord.
\end{enumerate}
\end{lemm}
Ce lemme implique en particulier qu'une surface polyédrale tendue à bord 
est contenue dans l'enveloppe convexe de son bord.

Enfin, un dernier énoncé, dû à B.~Grünbaum, permettra d'établir un lien
entre la dimension de l'espace ambiant et les graphes plongés dans 
la surface.
\begin{lemm}[\cite{gr65}] \label{rappel:lem5}
Le 1-squelette d'un polytope convexe de dimension~$d$ admet une
subdivision du graphe complet $K_{d+1}$ comme sous-graphe.
\end{lemm}
Une subdivision d'un graphe est un graphe obtenu par insertion de sommets
sur les arêtes. 
Ce dernier résultat admet plusieurs formulations. D'un point de vue 
topologique on peut dire que \emph{le graphe complet $K_{d+1}$ se plonge 
dans le 1-squelette de tout polytope convexe de dimension~$d$}. D'un
point de vue combinatoire, il signifie que \emph{le  1-squelette d'un 
polytope convexe de dimension~$d$ admet le graphe complet $K_{d+1}$
comme mineur}.

\section{Borne sur la codimension des plongement tendus}\label{borne}
L'objet de cette section est de démontrer le théorème~\ref{intro:th1}
en utilisant les résultats rappelés dans la section précédente.

 Considérons une surface compacte à bord $S$ plongée de manière
polyédrale et tendue et substantielle dans un espace $\R^n$ et notons 
$G$ le 1-squelette de son enveloppe convexe. Le graphe $G$ est contenu
dans $S$ selon le lemme~\ref{rappel:lem1}.

 Le lemme~\ref{rappel:lem4} assure que les sommets de $G$ sont situés 
sur le bord de $S$, c'est-à-dire que $G$ se plonge proprement dans $S$.

 Comme le plongement de $S$ est substantiel, son enveloppe convexe est
un polytope de dimension $n$. Le graphe $G$ contient donc une subdivision
du graphe complet à $n+1$ sommets $K_{n+1}$ d'après le lemme~\ref{rappel:lem5}. 
Ce graphe $K_{n+1}$ est donc lui-aussi proprement plongé dans $S$.

Enfin, le théorème~\ref{rappel:th} affirme que $c_0(S)$ est le
nombre du sommet du plus grand graphe complet qui se plonge proprement
dans $S$. On en déduit que $n+1\leq c_0$, ce qui conclut la démonstration
du théorème~\ref{intro:th1}.

\section{Construction de plongements tendus}\label{cons}
\subsection{Trois méthodes générales}
Ce paragraphe est consacré au rappel d'une méthode de construction
de plongement tendu partant d'une triangulation de la surface par
un graphe complet et à la démonstration de deux résultats sur la réalisation
tendue d'opérations topologiques sur une surface.

Nous commençons avec la construction d'un plongement tendu d'une surface
(avec ou sans bord) qu'on suppose triangulée par un graphe complet. Cette
technique, que nous qualifierons de \emph{construction canonique} était
déjà à la base des constructions de W.~Kühnel. 

\begin{theo}\label{cons:canonique}
Soit $S$ une surface compacte triangulée par un graphe complet $K_{n+1}$.
Si $S$ a un bord, on suppose de plus que $K_{n+1}$ est proprement
plongé dans $S$. Alors $S$ admet un plongement substantiel polyédral tendu 
dans $\R^n$.
\end{theo}
\begin{proof}
On considère un simplexe à $n+1$ sommets dans $\R^n$. Son 1-squelette,
qu'on identifiera au graphe $K_{n+1}$, est
0-tendu d'après le lemme~\ref{rappel:lem3}. On construit un plongement
de $S$ dans $\R^n$ en envoyant chaque face de la triangulation induite par
$K_{n+1}$ sur la 2-face correspondante du simplexe. Le 1-squelette 
de ce plongement de $S$ est le 1-squelette du simplexe. Le plongement est
donc 0-tendu selon le lemme~\ref{rappel:lem2}. Si la surface est close,
cela suffit à conclure que le plongement est tendu (\cite{ku95},
lemme~2.5). Si la surface a un bord, on remarque que l'enveloppe convexe
de son plongement est exactement le simplexe. Comme le graphe $K_{n+1}$ 
est proprement plongé, les points extrêmes de cette enveloppe convexe son
sur le bord de la surface, et on conclut grâce au lemme~\ref{rappel:lem4}.
\end{proof}

Nous en venons maintenant à la démonstration de deux résultats sur
la construction d'un plongement tendu d'une surface à partir du plongement
d'une autre surface. Le premier concerne les surfaces obtenues par adjonction
d'une anse sur le bord. Le second s'applique aux surfaces obtenues 
par suppression d'un disque et sera utilisé systématiquement dans la 
démonstration du théorème~\ref{intro:th2} pour 
les surfaces ayant un grand nombre de composantes de bord.
\begin{prop}\label{cons:anse}
Soient $S$ une surface polyédrale et $a$, $b$, $c$ et $d$ quatre points 
distincts de $S$ tels que $[ab]$ et $[cd]$ soient deux arêtes du bord de $S$. 

 Si le plongement de $S$ est 0-tendu (resp. tendu) et que l'intérieur du 
tétraèdre $abcd$ 
ne rencontre pas ce plongement, alors une surface $S'$ obtenue par 
l'adjonction d'une anse entre $[ab]$ et $[cd]$ admet un plongement polyédral
0-tendu (resp. tendu) dans le même espace.

Si de plus le plongement polyédral de $S$ a tous ses sommets
sur son bord, alors le plongement de $S'$ aussi.
\end{prop}
\begin{proof}
Le principe de la démonstration, illustré par la figure~\ref{cons:figanse}, 
consiste à ajouter l'anse dans le tétraèdre
$abcd$ en ajoutant deux sommets et quatre triangles à la surface $S$.
On va d'abord traiter le cas où $S$ est 0-tendu, en indiquant ensuite
comment le cas tendu s'en déduit.
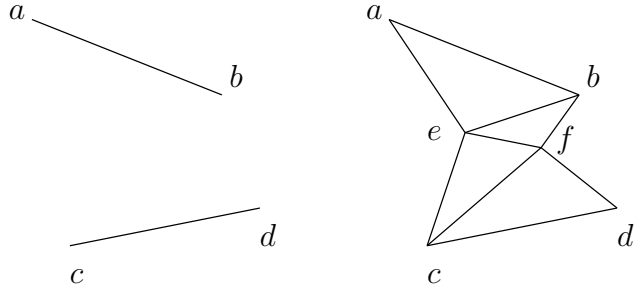
\begin{figure}[h]
\begin{center}
\begin{picture}(0,0)%
\includegraphics{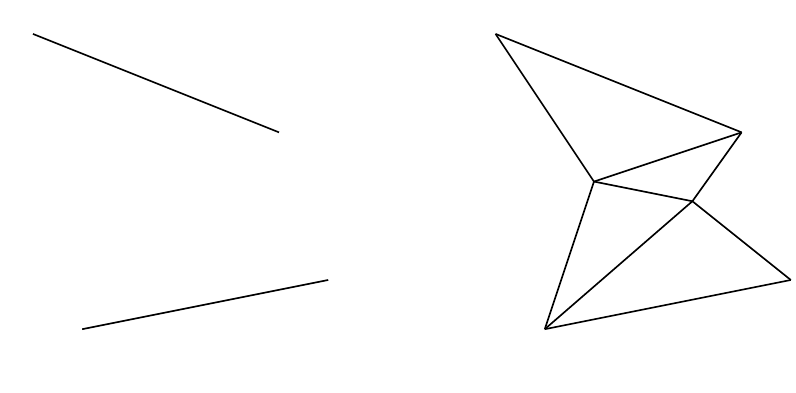}%
\end{picture}%
\setlength{\unitlength}{4144sp}%
\begingroup\makeatletter\ifx\SetFigFont\undefined%
\gdef\SetFigFont#1#2#3#4#5{%
  \reset@font\fontsize{#1}{#2pt}%
  \fontfamily{#3}\fontseries{#4}\fontshape{#5}%
  \selectfont}%
\fi\endgroup%
\begin{picture}(3630,1786)(886,-1025)
\put(2206,209){\makebox(0,0)[lb]{\smash{{\SetFigFont{12}{14.4}{\rmdefault}{\mddefault}{\updefault}{\color[rgb]{0,0,0}$b$}%
}}}}
\put(1261,-961){\makebox(0,0)[lb]{\smash{{\SetFigFont{12}{14.4}{\rmdefault}{\mddefault}{\updefault}{\color[rgb]{0,0,0}$c$}%
}}}}
\put(2386,-736){\makebox(0,0)[lb]{\smash{{\SetFigFont{12}{14.4}{\rmdefault}{\mddefault}{\updefault}{\color[rgb]{0,0,0}$d$}%
}}}}
\put(901,614){\makebox(0,0)[lb]{\smash{{\SetFigFont{12}{14.4}{\rmdefault}{\mddefault}{\updefault}{\color[rgb]{0,0,0}$a$}%
}}}}
\put(4321,209){\makebox(0,0)[lb]{\smash{{\SetFigFont{12}{14.4}{\rmdefault}{\mddefault}{\updefault}{\color[rgb]{0,0,0}$b$}%
}}}}
\put(3376,-961){\makebox(0,0)[lb]{\smash{{\SetFigFont{12}{14.4}{\rmdefault}{\mddefault}{\updefault}{\color[rgb]{0,0,0}$c$}%
}}}}
\put(4501,-736){\makebox(0,0)[lb]{\smash{{\SetFigFont{12}{14.4}{\rmdefault}{\mddefault}{\updefault}{\color[rgb]{0,0,0}$d$}%
}}}}
\put(3016,614){\makebox(0,0)[lb]{\smash{{\SetFigFont{12}{14.4}{\rmdefault}{\mddefault}{\updefault}{\color[rgb]{0,0,0}$a$}%
}}}}
\put(4141,-151){\makebox(0,0)[lb]{\smash{{\SetFigFont{12}{14.4}{\rmdefault}{\mddefault}{\updefault}{\color[rgb]{0,0,0}$f$}%
}}}}
\put(3376,-106){\makebox(0,0)[lb]{\smash{{\SetFigFont{12}{14.4}{\rmdefault}{\mddefault}{\updefault}{\color[rgb]{0,0,0}$e$}%
}}}}
\end{picture}%
\end{center}
\caption{Adjonction d'anse\label{cons:figanse}}
\end{figure}

Les deux nouveaux sommets $e$ et $f$ doivent être placés de manière à 
ce que la surface obtenue reste 0-tendue. On peut les définir précisément
en donnant leurs coordonnées barycentriques dans la base affine $(a,b,c,d)$.
Le point $e$ sera donné par les poids $(2,1,2,1)$ et le point $f$ par
les poids $(1,2,1,2)$. On peut vérifier que $e$ est dans l'enveloppe 
convexe de $a$, $c$ et $f$, et de même $f$ dans l'enveloppe convexe
de $b$, $d$ et $e$. La surface obtenue est bien plongée puisque $S$ ne 
rencontre pas l'intérieur de $abcd$, et d'après les lemmes~\ref{rappel:lem2}
et~\ref{rappel:lem3}, ce plongement est 0-tendu.

Si la surface $S$ est tendue, les points extrêmes de son enveloppe convexe
sont sur son bord selon le lemme~\ref{rappel:lem4}. Comme l'adjonction 
d'anse qu'on vient de construire n'ajoute pas de point extrême et que les
quatre points $a$, $b$, $c$ et $d$ restent sur le bord,
le plongement obtenu vérifie la même propriété et le plongement est donc tendu.
\end{proof}
\begin{rema}\label{cons:rem}
On peut changer la surface obtenue (en particulier son nombre de composantes
de bord) en permutant le rôle de $a$ et $b$. Par exemple, en partant du
ruban de Möbius, on peut obtenir soit le plan projectif privé de deux disques,
soit la bouteille de Klein privée d'un disque.
\end{rema}
 
On rappelle que si $S$ est une surface close, on note $S_p$ la surface
$S$ privée de $p$~disques disjoints.
\begin{prop}\label{cons:trou}
Soit $S$ une surface close. Si $S_p$ admet un plongement substantiel 0-tendu 
(resp. tendu) dans $\R^n$, alors $S_{p+1}$ aussi.
Si de plus le plongement polyédral de $S_p$ a tous ses sommets
sur son bord, alors le plongement de $S_{p+1}$ aussi.
\end{prop}
\begin{proof}
Topologiquement, l'opération consiste à enlever un disque à la surface
$S_p$. On la réalise sur le plongement en enlevant un triangle dans
une face du plongement comme sur la figure~\ref{cons:figtrou} : en ajoutant
trois sommets dans une face trianglulaire du plongement de $S_p$, on découpe
cette face en sept triangles, puis on enlève le triangle central (si le 
plongement de $S_p$ n'a pas de face triangulaire, on peut en créer une en
ajoutant une arête dans une face quelconque).
\begin{figure}[h]
\begin{center}
\begin{picture}(0,0)%
\includegraphics{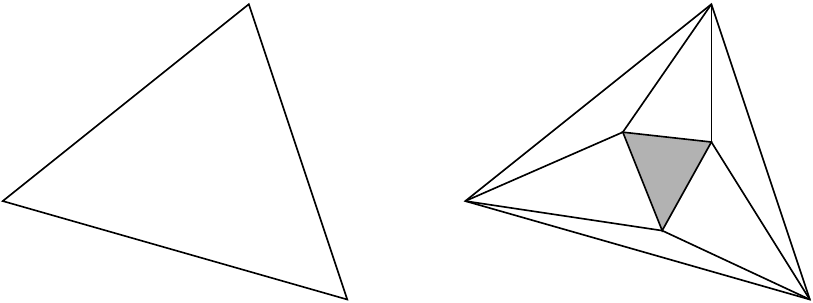}%
\end{picture}%
\setlength{\unitlength}{4144sp}%
\begingroup\makeatletter\ifx\SetFigFont\undefined%
\gdef\SetFigFont#1#2#3#4#5{%
  \reset@font\fontsize{#1}{#2pt}%
  \fontfamily{#3}\fontseries{#4}\fontshape{#5}%
  \selectfont}%
\fi\endgroup%
\begin{picture}(3714,1374)(889,-1198)
\end{picture}%
\end{center}
\caption{Création d'une composante de bord\label{cons:figtrou}}
\end{figure}

On obtient ainsi un plongement polyédral de $S_{p+1}$. En outre, 
il est clair que les nouveaux sommets sont dans 
l'enveloppe convexe de leurs voisins, ce plongement est donc 0-tendu d'après
le lemme~\ref{rappel:lem3}. En outre, ces sommets sont sur le bord de 
$S_{p+1}$.

Le cas où le plongement de $S_p$ est tendu se traite comme dans la proposition
précédente.
\end{proof}
\begin{rema}\label{cons:trou2}
Si tous les sommets du triangle initial de la figure~\ref{cons:figtrou}
ne sont pas sur le bord de la surface, on peut modifier la construction
de manière à en placer un ou deux sur le bord, comme sur la 
figure~\ref{cons:figtrou2}.
\end{rema}
\begin{figure}[h]
\begin{center}
\begin{picture}(0,0)%
\includegraphics{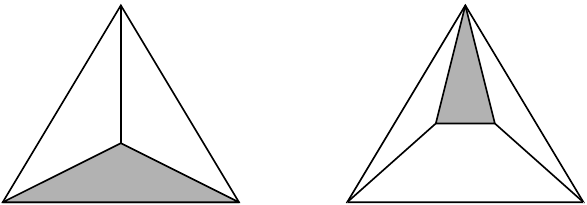}%
\end{picture}%
\setlength{\unitlength}{4144sp}%
\begingroup\makeatletter\ifx\SetFigFont\undefined%
\gdef\SetFigFont#1#2#3#4#5{%
  \reset@font\fontsize{#1}{#2pt}%
  \fontfamily{#3}\fontseries{#4}\fontshape{#5}%
  \selectfont}%
\fi\endgroup%
\begin{picture}(2679,924)(889,-523)
\end{picture}%
\end{center}
\caption{Placement d'un ou deux sommets sur le bord\label{cons:figtrou2}}
\end{figure}

\subsection{Démonstration du théorème~\ref{intro:th2}}
 Il s'agit maintenant d'exhiber des plongements tendus pour toutes les
surfaces annoncées par le théorème~\ref{intro:th2}. On va traiter 
les différentes surfaces closes $S$ par caractéristique d'Euler décroissante.

\subsubsection{La sphère $\mathbb S^2$}
Comme $\mathbb S_1^2$ est un disque
et que $c_0(\mathbb S_1^2)=3$, il suffit de plonger $\mathbb S^2_1$ dans 
$\R^2$ sous la forme d'un triangle. 

La surface $\mathbb S^2_2$ est un cylindre, qu'on peut plonger de manière
tendue dans $\R^3$ en le décomposant en trois rectangles comme sur la 
figure~\ref{cons:figcyl}.
\begin{figure}[h]
\begin{center}
\begin{picture}(0,0)%
\includegraphics{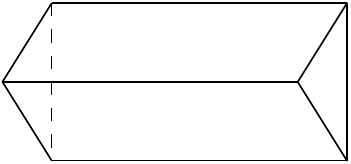}%
\end{picture}%
\setlength{\unitlength}{4144sp}%
\begingroup\makeatletter\ifx\SetFigFont\undefined%
\gdef\SetFigFont#1#2#3#4#5{%
  \reset@font\fontsize{#1}{#2pt}%
  \fontfamily{#3}\fontseries{#4}\fontshape{#5}%
  \selectfont}%
\fi\endgroup%
\begin{picture}(1599,744)(1564,-568)
\end{picture}%
\end{center}
\caption{Le cylindre\label{cons:figcyl}}
\end{figure}

Le cas des surfaces $\mathbb S^2_p$, $p\geq3$ se traite en appliquant 
la proposition~\ref{cons:trou}.

\subsubsection{Le plan projectif $\mathbb P^2$}
La surface $\mathbb P^2_1$ est le ruban de Möbius, qui est triangulé par
le graphe $K_5$ comme le montre la figure~\ref{cons:figm}.
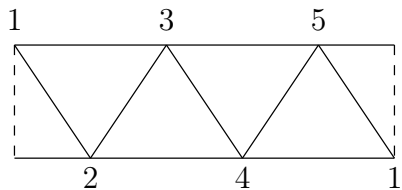
\begin{figure}[h]
\begin{center}
\begin{picture}(0,0)%
\includegraphics{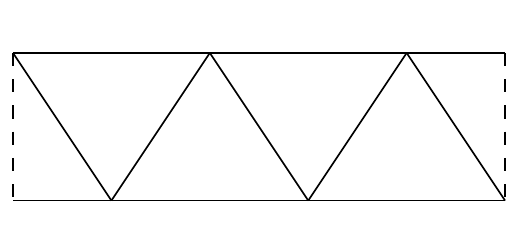}%
\end{picture}%
\setlength{\unitlength}{4144sp}%
\begingroup\makeatletter\ifx\SetFigFont\undefined%
\gdef\SetFigFont#1#2#3#4#5{%
  \reset@font\fontsize{#1}{#2pt}%
  \fontfamily{#3}\fontseries{#4}\fontshape{#5}%
  \selectfont}%
\fi\endgroup%
\begin{picture}(2387,1110)(391,-481)
\put(406,479){\makebox(0,0)[lb]{\smash{{\SetFigFont{12}{14.4}{\rmdefault}{\mddefault}{\updefault}{\color[rgb]{0,0,0}1}%
}}}}
\put(1306,479){\makebox(0,0)[lb]{\smash{{\SetFigFont{12}{14.4}{\rmdefault}{\mddefault}{\updefault}{\color[rgb]{0,0,0}3}%
}}}}
\put(1756,-466){\makebox(0,0)[lb]{\smash{{\SetFigFont{12}{14.4}{\rmdefault}{\mddefault}{\updefault}{\color[rgb]{0,0,0}4}%
}}}}
\put(2206,479){\makebox(0,0)[lb]{\smash{{\SetFigFont{12}{14.4}{\rmdefault}{\mddefault}{\updefault}{\color[rgb]{0,0,0}5}%
}}}}
\put(2656,-466){\makebox(0,0)[lb]{\smash{{\SetFigFont{12}{14.4}{\rmdefault}{\mddefault}{\updefault}{\color[rgb]{0,0,0}1}%
}}}}
\put(856,-466){\makebox(0,0)[lb]{\smash{{\SetFigFont{12}{14.4}{\rmdefault}{\mddefault}{\updefault}{\color[rgb]{0,0,0}2}%
}}}}
\end{picture}%
\end{center}
\caption{Le ruban de Möbius\label{cons:figm}}
\end{figure}
On peut donc mettre en \oe{}uvre la construction canonique du 
théorème~\ref{cons:canonique}. On obtient un plongement polyédral 
substantiel tendu du ruban de Möbius dans $\R^4$, les sommets
étant bien sur le bord de la surface, comme sur la figure~\ref{cons:figm}.

Le plongement tendu de $\mathbb P^2_2$ dans $\R^4$ se construit à partir de
celui du ruban de Möbius en appliquant l'une des propositions~\ref{cons:anse} 
(cf. remarque~\ref{cons:rem}) ou~\ref{cons:trou}.

Le plan projectif $\mathbb P^2$ est triangulé par un plongement du
graphe complet $K_6$ (c'est la projectivisation de l'icosaèdre). 
On obtient donc un plongement tendu de $\mathbb P^2_3$ dans $\R^5$
en partant du plongement canonique de $\mathbb P^2$ et en appliquant
trois fois la remarque~\ref{cons:trou2}. Les plongements tendus de
$\mathbb P^2_p$, $p\geq4$ s'obtiennent grâce à la proposition~\ref{cons:trou}.

\subsubsection{Le tore $\mathbb T^2$}
Le tore est triangulé par un plongement de $K_7$ (cette observation semble
remonter à A.~F.~Möbius \cite{mo86}). En enlevant un sommet
et les faces adjacentes, on obtient une triangulation de $\mathbb T^2_1$
par le graphe $K_6$. La construction canonique fournit un plongement tendu
de $\mathbb T^2_1$ dans~$\R^5$, tous les sommets étant sur le bord.

La surface $\mathbb T^2_2$ est traitée en appliquant l'une des 
propositions~\ref{cons:anse} ou~\ref{cons:trou}. 

Pour les surfaces $\mathbb T^2_p$, $p\geq3$, on part du plongement
canonique de $\mathbb T^2$ dans $\R^6$, on place six sommets sur le bord
en enlevant deux faces et le septième en appliquant la 
remarque~\ref{cons:trou2}. On peut ensuite appliquer la 
proposition~\ref{cons:trou}.
\subsubsection{La bouteille de Klein $\mathbb K^2$}
La surface $\mathbb K^2_1$ s'obtient par adjonction d'anse à partir
du ruban de Möbius. Son plongement tendu dans $\R^4$ est donc donné
par la proposition~\ref{cons:anse}.

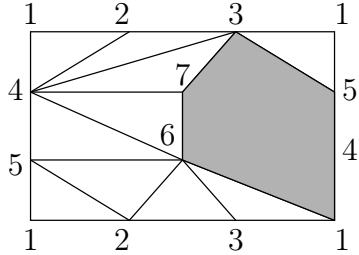
\begin{figure}[h]
\begin{center}
\begin{picture}(0,0)%
\includegraphics{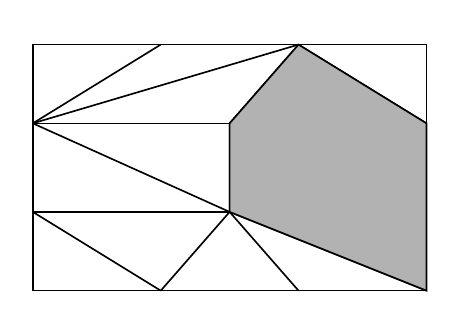}%
\end{picture}%
\setlength{\unitlength}{4144sp}%
\begingroup\makeatletter\ifx\SetFigFont\undefined%
\gdef\SetFigFont#1#2#3#4#5{%
  \reset@font\fontsize{#1}{#2pt}%
  \fontfamily{#3}\fontseries{#4}\fontshape{#5}%
  \selectfont}%
\fi\endgroup%
\begin{picture}(2117,1515)(301,-931)
\put(406,434){\makebox(0,0)[lb]{\smash{{\SetFigFont{12}{14.4}{\rmdefault}{\mddefault}{\updefault}{\color[rgb]{0,0,0}1}%
}}}}
\put(1621,434){\makebox(0,0)[lb]{\smash{{\SetFigFont{12}{14.4}{\rmdefault}{\mddefault}{\updefault}{\color[rgb]{0,0,0}3}%
}}}}
\put(946,434){\makebox(0,0)[lb]{\smash{{\SetFigFont{12}{14.4}{\rmdefault}{\mddefault}{\updefault}{\color[rgb]{0,0,0}2}%
}}}}
\put(2251,434){\makebox(0,0)[lb]{\smash{{\SetFigFont{12}{14.4}{\rmdefault}{\mddefault}{\updefault}{\color[rgb]{0,0,0}1}%
}}}}
\put(1621,-916){\makebox(0,0)[lb]{\smash{{\SetFigFont{12}{14.4}{\rmdefault}{\mddefault}{\updefault}{\color[rgb]{0,0,0}3}%
}}}}
\put(2251,-916){\makebox(0,0)[lb]{\smash{{\SetFigFont{12}{14.4}{\rmdefault}{\mddefault}{\updefault}{\color[rgb]{0,0,0}1}%
}}}}
\put(406,-916){\makebox(0,0)[lb]{\smash{{\SetFigFont{12}{14.4}{\rmdefault}{\mddefault}{\updefault}{\color[rgb]{0,0,0}1}%
}}}}
\put(2296,-376){\makebox(0,0)[lb]{\smash{{\SetFigFont{12}{14.4}{\rmdefault}{\mddefault}{\updefault}{\color[rgb]{0,0,0}4}%
}}}}
\put(316,-16){\makebox(0,0)[lb]{\smash{{\SetFigFont{12}{14.4}{\rmdefault}{\mddefault}{\updefault}{\color[rgb]{0,0,0}4}%
}}}}
\put(316,-466){\makebox(0,0)[lb]{\smash{{\SetFigFont{12}{14.4}{\rmdefault}{\mddefault}{\updefault}{\color[rgb]{0,0,0}5}%
}}}}
\put(2296,-16){\makebox(0,0)[lb]{\smash{{\SetFigFont{12}{14.4}{\rmdefault}{\mddefault}{\updefault}{\color[rgb]{0,0,0}5}%
}}}}
\put(1306, 74){\makebox(0,0)[lb]{\smash{{\SetFigFont{12}{14.4}{\rmdefault}{\mddefault}{\updefault}{\color[rgb]{0,0,0}7}%
}}}}
\put(1216,-286){\makebox(0,0)[lb]{\smash{{\SetFigFont{12}{14.4}{\rmdefault}{\mddefault}{\updefault}{\color[rgb]{0,0,0}6}%
}}}}
\put(946,-916){\makebox(0,0)[lb]{\smash{{\SetFigFont{12}{14.4}{\rmdefault}{\mddefault}{\updefault}{\color[rgb]{0,0,0}2}%
}}}}
\end{picture}%
\end{center}
\caption{La bouteille de Klein\label{cons:figklein}}
\end{figure}
Le plongement tendu de $\mathbb K^2_2$ dans $\R^5$ va s'obtenir en plaçant
tous les sommets sauf un sur un bord, puis en appliquant la 
remarque~\ref{cons:trou2}.
Le graphe complet $K_6$ se plonge dans la bouteille de Klein, par exemple
comme sur la figure~\ref{cons:figklein}. Ce plongement découpe la surface
en 10 triangles et un hexagone ayant deux sommets identiques. 
En découpant le pentagone 14536, on obtient une triangulation avec 5 sommets 
sur le bord de la surface mais on crée une arête 36 qui apparaît déjà 
ailleurs, on ne peut donc pas construire directement un plongement de
cette manière. On ajoute donc un sommet 7 dans le triangle 346 de sorte
que le bord du quadrilatère 3467 ne rencontre pas d'autre arête. La
remarque~\ref{cons:trou2} permet de placer le sommet~2 sur un autre bord.

Pour $\mathbb K^2_p$, $p\geq3$, on fait encore appel à la 
proposition~\ref{cons:trou}.

\subsubsection{La surface de caractéristique -1}
En notant $S$ la surface close de caractéristique d'Euler égale à -1,
un plongement tendu de la surface $S_1$ peut s'obtenir à partir de 
$\mathbb T^2_1$ par adjonction d'une anse non orientable.

Le cas de $S_2$ se base sur un plongement de $K_7$ dans $S$ 
(figure~\ref{cons:fign3}, inspirée de la figure~4, p.~20 de \cite{ku95})
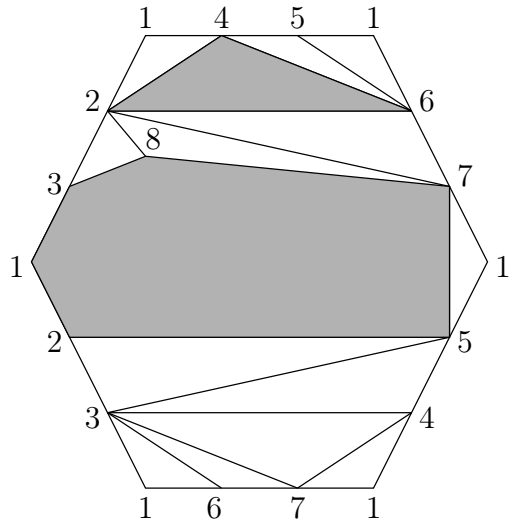
\begin{figure}[h]
\begin{center}
\begin{picture}(0,0)%
\includegraphics{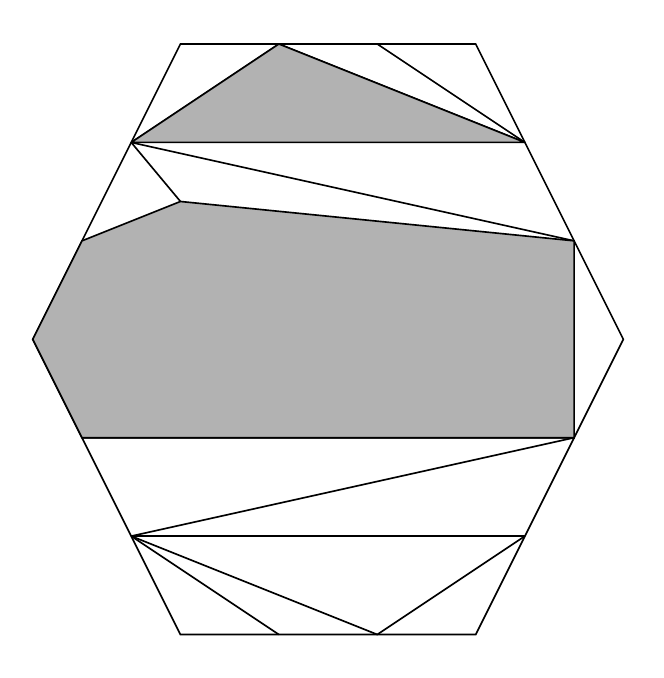}%
\end{picture}%
\setlength{\unitlength}{4144sp}%
\begingroup\makeatletter\ifx\SetFigFont\undefined%
\gdef\SetFigFont#1#2#3#4#5{%
  \reset@font\fontsize{#1}{#2pt}%
  \fontfamily{#3}\fontseries{#4}\fontshape{#5}%
  \selectfont}%
\fi\endgroup%
\begin{picture}(3017,3090)(301,-2281)
\put(316,-826){\makebox(0,0)[lb]{\smash{{\SetFigFont{12}{14.4}{\rmdefault}{\mddefault}{\updefault}{\color[rgb]{0,0,0}1}%
}}}}
\put(3196,-826){\makebox(0,0)[lb]{\smash{{\SetFigFont{12}{14.4}{\rmdefault}{\mddefault}{\updefault}{\color[rgb]{0,0,0}1}%
}}}}
\put(1081,659){\makebox(0,0)[lb]{\smash{{\SetFigFont{12}{14.4}{\rmdefault}{\mddefault}{\updefault}{\color[rgb]{0,0,0}1}%
}}}}
\put(2431,659){\makebox(0,0)[lb]{\smash{{\SetFigFont{12}{14.4}{\rmdefault}{\mddefault}{\updefault}{\color[rgb]{0,0,0}1}%
}}}}
\put(2431,-2266){\makebox(0,0)[lb]{\smash{{\SetFigFont{12}{14.4}{\rmdefault}{\mddefault}{\updefault}{\color[rgb]{0,0,0}1}%
}}}}
\put(766,164){\makebox(0,0)[lb]{\smash{{\SetFigFont{12}{14.4}{\rmdefault}{\mddefault}{\updefault}{\color[rgb]{0,0,0}2}%
}}}}
\put(766,-1726){\makebox(0,0)[lb]{\smash{{\SetFigFont{12}{14.4}{\rmdefault}{\mddefault}{\updefault}{\color[rgb]{0,0,0}3}%
}}}}
\put(2746,-1726){\makebox(0,0)[lb]{\smash{{\SetFigFont{12}{14.4}{\rmdefault}{\mddefault}{\updefault}{\color[rgb]{0,0,0}4}%
}}}}
\put(1531,659){\makebox(0,0)[lb]{\smash{{\SetFigFont{12}{14.4}{\rmdefault}{\mddefault}{\updefault}{\color[rgb]{0,0,0}4}%
}}}}
\put(1981,659){\makebox(0,0)[lb]{\smash{{\SetFigFont{12}{14.4}{\rmdefault}{\mddefault}{\updefault}{\color[rgb]{0,0,0}5}%
}}}}
\put(2971,-1276){\makebox(0,0)[lb]{\smash{{\SetFigFont{12}{14.4}{\rmdefault}{\mddefault}{\updefault}{\color[rgb]{0,0,0}5}%
}}}}
\put(2746,164){\makebox(0,0)[lb]{\smash{{\SetFigFont{12}{14.4}{\rmdefault}{\mddefault}{\updefault}{\color[rgb]{0,0,0}6}%
}}}}
\put(1486,-2266){\makebox(0,0)[lb]{\smash{{\SetFigFont{12}{14.4}{\rmdefault}{\mddefault}{\updefault}{\color[rgb]{0,0,0}6}%
}}}}
\put(1981,-2266){\makebox(0,0)[lb]{\smash{{\SetFigFont{12}{14.4}{\rmdefault}{\mddefault}{\updefault}{\color[rgb]{0,0,0}7}%
}}}}
\put(2971,-286){\makebox(0,0)[lb]{\smash{{\SetFigFont{12}{14.4}{\rmdefault}{\mddefault}{\updefault}{\color[rgb]{0,0,0}7}%
}}}}
\put(1126,-61){\makebox(0,0)[lb]{\smash{{\SetFigFont{12}{14.4}{\rmdefault}{\mddefault}{\updefault}{\color[rgb]{0,0,0}8}%
}}}}
\put(541,-331){\makebox(0,0)[lb]{\smash{{\SetFigFont{12}{14.4}{\rmdefault}{\mddefault}{\updefault}{\color[rgb]{0,0,0}3}%
}}}}
\put(1081,-2266){\makebox(0,0)[lb]{\smash{{\SetFigFont{12}{14.4}{\rmdefault}{\mddefault}{\updefault}{\color[rgb]{0,0,0}1}%
}}}}
\put(541,-1276){\makebox(0,0)[lb]{\smash{{\SetFigFont{12}{14.4}{\rmdefault}{\mddefault}{\updefault}{\color[rgb]{0,0,0}2}%
}}}}
\end{picture}%
\end{center}
\caption{La surface de caractéristique -1\label{cons:fign3}}
\end{figure}
Ce plongement découpe $S$ en 14 triangles et un hexagone ayant deux 
sommets identiques. On procède alors comme pour $\mathbb K^2_2$: on applique
la construction canonique à la surface $S$ privée du pentagone 12573, on
ajoute un sommet dans le triangle 237 de manière à obtenir un plongement
et on enlève le triangle 246 de manière à ce que tous les sommets soient
sur le bord.

Le cas de $S_p$, $p\geq3$ se traite comme précédemment.

\subsubsection{La surface orientable de genre 2}

Pour traiter le cas de la surface $S=\mathbb T^2\#\mathbb T^2$, on va utiliser
une description combinatoire d'un plongement de $K_8$ dans cette surface
tirée de \cite{ri74} (section~2.2, p.~23). Cette description est donnée 
par la table~\ref{cons:table} qui se lit comme suit : chaque ligne commence
par l'indice d'un des sommets du graphe et donne ensuite la liste de ses
voisins ordonnée circulairement, le sens de l'ordre étant induit par 
l'orientation de la surface. 
\begin{table}[h]
\begin{center}
\renewcommand{\arraystretch}{1.5}
\begin{tabular}{ccccccccc}
0.&2&7&3&1&4&5&6\\
2.&4&1&5&3&6&7&0\\
4.&6&3&7&5&0&1&2\\
6.&0&5&2&7&2&3&4\\
1.&7&6&5&2&4&0&3\\
3.&1&0&7&4&6&2&5\\
5.&3&2&1&6&0&4&7\\
7.&5&4&3&0&2&6&1
\end{tabular}
\medskip
\renewcommand{\arraystretch}{1}
\caption{Plongement du graphe $K_8$ dans $\mathbb T^2\#\mathbb T^2$}
\label{cons:table}
\end{center}
\end{table}

On peut vérifier que ce plongement contient deux quadrilatères, de sommets
0246 et 1357, les autres faces étant des triangles. Cette remarque permet
de conclure immédiatement pour les surfaces $S_p$, $p\geq2$: si on enlève
les deux quadrilatères, on obtient la surface $S_2$, qui est alors triangulée
par le graphe $K_8$, tous les sommets étant sur le bord. On peut alors
utiliser la construction canonique. Les valeurs supérieures de $p$ sont 
couvertes par la proposition~\ref{cons:trou}.

Passons maintenant au cas $p=1$. Le sommet~0 est adjacent à six triangles
et au quadrilatère 0246. Si on envève ces 7 faces, deux problèmes
empêchent d'appliquer la construction canonique : d'une part, le bord obtenu 
est l'octogone 27314564, où le sommet~4 apparaît deux fois ; d'autre part,
il reste une face qui n'est pas triangulaire, le quadrilatère 1357.

Le premier problème se traite comme dans le cas de la bouteille de Klein :
on ajoute un sommet 8 dans le triangle 246 contenu dans le quadrilatère 
0246. Au lieu d'enlever le quadrilatère 0246 dans son ensemble, on enlève
seulement le quadrilatère 0286 qu'il contient.

 Le second problème n'est pas apparu avec les surfaces précédentes.
Pour le résoudre, on ajoute un sommet 9 à l'intérieur du quadrilatère 1357
de manière à le trianguler, et on enlève le triangle 139. Le plongement
tendu est obtenu en appliquant la construction canonique en dehors des
deux quadrilatères, on place le sommet~8 comme indiqué précédemment et
on ajoute le sommet~9 à l'intérieur du tétraèdre 1357. Le bord 
est alors l'ennéagone 273914568. Les sommets~1 à~7 sont des points 
extrêmes de l'enveloppe convexe et les sommets~8 et~9 sont bien dans 
l'enveloppe convexe de leurs voisins. On obtient donc bien un plongement
tendu.

\bibliographystyle{smfalpha}
\bibliography{polyedral}

\end{document}